\documentclass[a4paper,11pt,reqno]{amsart}

\usepackage[utf8]{inputenc}

\usepackage{etex}

\usepackage{xcolor}

\definecolor{verydarkblue}{rgb}{0,0,0.5}

\usepackage[
    breaklinks,
    colorlinks,
    citecolor=verydarkblue,
    linkcolor=verydarkblue,
    urlcolor=verydarkblue,
    pagebackref=true,
    hyperindex
]{hyperref}

\backrefenglish

\usepackage{fancyhdr}

\usepackage[
    hscale=0.7,
    vscale=0.75,
    headheight=13pt,
    centering,
]{geometry}

\usepackage{amsmath}
\usepackage{amsthm}
\usepackage{amssymb}
\usepackage{mathtools}
\usepackage{stmaryrd}
\usepackage{mathdots}
\usepackage{framed}
\usepackage[capitalize]{cleveref}
\usepackage{array}
\usepackage[alphabetic]{amsrefs}
\usepackage[all,cmtip]{xy}
\usepackage{enumitem}
\usepackage{calligra,mathrsfs}

\theoremstyle{plain}

\newtheorem{introtheorem}{Theorem}

\crefname{introtheorem}{Theorem}{Theorems}

\crefname{introproposition}{Proposition}{Propositions}

\newtheorem{introcorollary}[introtheorem]{Corollary}

\crefname{introcorollary}{Corollary}{Corollaries}

\crefname{introconjecture}{Conjecture}{Conjectures}

\newtheorem{theorem}{Theorem}[section]
\newtheorem{proposition}[theorem]{Proposition}
\newtheorem{lemma}[theorem]{Lemma}
\newtheorem{corollary}[theorem]{Corollary}

\newtheorem{conjecture}[theorem]{Conjecture}

\theoremstyle{definition}

\newtheorem{definition}[theorem]{Definition}

\theoremstyle{remark}

\newtheorem{remark}[theorem]{Remark}

\newtheorem{example}[theorem]{Example}

\numberwithin{figure}{section}

\numberwithin{equation}{section}

\IfFileExists{./article-style.tex}{\input{article-style.tex}}{}

\usepackage{marginnote}

\def\Z{{\mathbb Z}}
\def\Q{{\mathbb Q}}

\def\C{{\mathbb C}}

\def\P{{\mathbb P}}

\def\cA{\mathcal{A}}

\def\cC{\mathcal{C}}

\def\cE{\mathcal{E}}
\def\cF{\mathcal{F}}

\def\cI{\mathcal{I}}

\def\cL{\mathcal{L}}
\def\cM{\mathcal{M}}
\def\cN{\mathcal{N}}

\def\cU{\mathcal{U}}
\def\cV{\mathcal{V}}

\def\O{\mathcal{O}}

\def\p{\pi}
\def\r{\rho}
\def\s{\sigma}

\def\.{\cdot}
\def\^{\widehat}
\def\~{\widetilde}

\def\ov{\overline}

\def\({\left(}
\def\){\right)}

\def\*{{}^*}

\renewcommand{\and}{ \ \ \text{ and } \ \ }

\def\an{\mathrm{an}}

\DeclareMathOperator{\Spec} {Spec}
\DeclareMathOperator{\BroadSpec}{\bold{Spec}}

\DeclareMathOperator{\Pic} {Pic}

\DeclareMathOperator{\NS} {NS}

\DeclareMathOperator{\Sym} {Sym}
\DeclareMathOperator{\Supp} {Supp}

\DeclareMathOperator{\Hom} {Hom}

\DeclareMathOperator{\cd} {cd}

\DeclareMathOperator{\Lef} {Lef}
\DeclareMathOperator{\Leff} {Leff}
\DeclareMathOperator{\Alb}{Alb}
\DeclareMathOperator{\alb}{alb}
\DeclareMathOperator{\Ker}{Ker}
\DeclareMathOperator{\Coker}{Coker}
\DeclareMathOperator{\sheafhom}{\mathscr{H}\text{\kern -3pt {\calligra\large om}}\,}
\DeclareMathOperator{\PicS}{\mathcal{P}\!\!\;\textit{ic}}
\DeclareMathOperator{\PicV}{\mathcal{P}\!\!\;\textit{ic}^0}

\DeclareMathOperator{\charK}{char}

\makeatletter
\@namedef{subjclassname@2020}{%
  \textup{2020} Mathematics Subject Classification}
\makeatother

\begin{document}

\title{Grothendieck--Lefschetz for ample subvarieties}

\author{Tommaso de Fernex}

\address{Department of Mathematics, University of Utah, Salt Lake City, UT 48112, USA}

\email{defernex@math.utah.edu}

\author{Chung Ching Lau}

\email{malccad@gmail.com}


\thanks{%
The research of the first author was partially supported by NSF Grant DMS-1700769
and by NSF Grant DMS-1440140 while in residence at  
MSRI in Berkeley during the Spring 2019 semester.
The research of the second author was partially supported by a Croucher Foundation Fellowship. 
}

\begin{abstract}
We establish a Grothendieck--Lefschetz theorem for smooth ample subvarieties
of smooth projective varieties over an algebraically closed field of characteristic zero and, more
generally, for smooth subvarieties whose complement has small cohomological dimension. 
A weaker statement is also proved in a more general context and in all characteristics. 
Several applications are included. 
\end{abstract}

\maketitle

\section{Introduction}

The notion of ample subscheme, which finds its roots in the work of Hartshorne \cite{Har70},
was recently formalized by Ottem in \cite{Ott12} where
a closed subscheme $Y$ of codimension $r$ of a projective variety $X$ of characteristic zero 
is said to be \emph{ample} if the exceptional divisor of the blow-up of $X$ along $Y$ is
$(r-1)$-ample in the sense of \cite{Tot13}. 
If $r = 1$ and $X$ is normal, this is equivalent to $Y$ being an effective ample Cartier divisor.
In higher codimensions, examples are given by subschemes defined by the vanishing
of regular sections of ample vector bundles. The assumption on the characteristic is relevant, 
and in fact it remains unclear at the moment what should be the correct definition of ample subscheme 
in positive characteristics. 

Several properties of ample subschemes are established in \cite{Ott12}.
The connection between Ottem's definition of ampleness
and the treatment in \cite{Har70} is manifested in the property stating 
that a locally complete intersection subscheme $Y$ 
of codimension $r$ of a smooth variety $X$ is ample if and only if the normal bundle $\cN_{Y/X}$ is ample
and the complement has cohomological dimension $\cd(X \setminus Y) = r-1$. 
Here, cohomological dimension is intended with respect to coherent cohomology.
Ottem deduces from this property a version 
of the Lefschetz hyperplane theorem which states that if $X$ is a smooth complex projective variety 
and $Y$ is an ample locally complete intersection subscheme, then 
$H^i(X,\Q) \to H^i(Y,\Q)$ is an isomorphism for $i < \dim Y$ and 
injective for $i = \dim Y$. It is important to remark that, 
differently from the case of ample divisors or, more generally, 
of schemes defined by 
regular sections of ample vector bundles, the Lefschetz hyperplane theorem
with integral coefficients can fail for ample subvarieties, even assuming that $Y$ is smooth.

In this paper, we establish a version of the
Grothendieck--Lefschetz theorem for ample subvarieties
and, more generally, for subvarieties whose complements
have small cohomological dimension. 
Our main result is the following.

\begin{introtheorem}[cf.\ \cref{th:Lefschetz}]
\label{th:Pic-intro}
Let $X$ be a smooth $n$-dimensional 
projective variety defined over an algebraically closed field of characteristic zero, 
and let $Y \subset X$ be a smooth subvariety.
\begin{enumerate}
\item
\label{item1:Pic-intro}
If $\cd(X \setminus Y) \le n - 3$, then the restriction map $\Pic(X) \to \Pic(Y)$ is injective
and induces an isomorphism $\Pic^0(X) \cong \Pic^0(Y)$.
\item
\label{item2:Pic-intro}
If $\cd(X \setminus Y) < n - 3$, then the map $\Pic(X) \to \Pic(Y)$ has finite cokernel.
\end{enumerate}
\end{introtheorem}

Case \eqref{item1:Pic-intro} of the theorem
applies to ample subvarieties of dimension $\ge 2$, and case \eqref{item2:Pic-intro}
to ample subvarieties of dimension $\ge 3$. 
The example of Veronese embeddings shows that the cokernel of $\Pic(X) \to \Pic(Y)$ can be nontrivial
(see \cref{eg:veronese}), hence the conclusions of \cref{th:Pic-intro} are optimal.
By duality, it follows from case \eqref{item1:Pic-intro} of \cref{th:Pic-intro} that the natural map
on Albanese varieties $\Alb(Y) \to \Alb(X)$ is an isomorphism, and this implies that the restriction of
the Albanese morphism of $X$ to $Y$ is the Albanese morphism of $Y$
(cf.\ \cref{th:Alb}).

The main new ingredient in the proof of \cref{th:Pic-intro} is the following 
torsion-freeness property which holds in all characteristics and under weaker assumptions on $X$ and $Y$.

\begin{introtheorem}[cf.\ \cref{cor:tf}]
\label{th:tf-intro}
Let $X$ be an $n$-dimensional projective variety defined over an algebraically closed field $k$, 
and let $Y\subset X$ be a positive-dimensional closed subscheme. 
Assume that $\cd(X \setminus Y) < n-1$.
Then the kernel of the restriction map $\Pic(X)\to\Pic(Y)$ is torsion-free if $\charK k = 0$, 
and its torsion is killed by a power of $p$ if $\charK k = p > 0$. 
\end{introtheorem}

We deduce \cref{th:tf-intro} from an analogous statement (see \cref{prop:tf-K(^X)}) 
where the assumption is expressed in terms of the ring $K(\^X)$ of formal rational functions of $X$ along $Y$. 
This ring was first introduced in \cite{HM68} as
a tool to measure the positivity of a closed subscheme, 
where $Y$ is said to be $G2$ if $K(\^X)$ is a finite module over $K(X)$ 
and $G3$ if $K(\^X) = K(X)$.

We now present some applications of these results.
The first two application holds in arbitrary characteristic, over an arbitrary algebraically closed field.

\begin{introcorollary}[cf.\ \cref{cor:G2}]
\label{cor:G2-intro}
If $X$ is a normal projective variety 
and $Y\subset X$ is a positive-dimensional connected closed subscheme that is $G2$ in $X$, 
then the induced map on Albanese varieties $\Alb(Y)\to \Alb(X)$ is surjective.
\end{introcorollary}

In characteristic zero, \cref{cor:G2-intro} can be seen as a strengthening of 
a theorem attributed to Matsumura where the surjectivity of the map between Albanese varieties is
established under the more restrictive assumptions
that both $X$ and $Y$ are smooth and $Y$ has ample normal bundle 
(cf.\ \cite[Exercise~III.4.15]{Har70}).

\begin{introcorollary}[cf.\ \cref{cor:diagonal}]
\label{cor:diagonal-intro}
If $X$ is a projective variety with nontrivial Albanese variety, 
then the diagonal of $X\times X$ is not $G2$.
\end{introcorollary}

This should be contrasted with a result of 
B\u{a}descu and Schneider stating that if $X$ is a complex rational homogeneous space 
then the diagonal of $X\times X$ is $G3$ \cite[Theorem~2]{BS96}, 	
and a result of Halic stating that if $X$ is a 
smooth projective variety with non-pseudoeffective cotangent bundle 
then the diagonal is $G2$ \cite[Proposition~3.10]{Hal19}.
\cref{cor:diagonal-intro} shows that the result of Halic is fairly close to optimal.

We now turn to the case of characteristic zero; we still assume that the field
is algebraically closed. The following easy application of \cref{th:Pic-intro} 
generalizes a result of Sommese stating that abelian varieties of dimension $\ge 2$
cannot be realized as ample divisors of smooth projective varieties \cite[Corollary~I-A]{Som76}.

\begin{introcorollary}[cf.\ \cref{th:Y=A-cannot-be-ample}]
Abelian varieties of dimension $\ge 2$ cannot be realized as ample subvarieties 
of any smooth projective variety.
\end{introcorollary}

\cref{th:tf-intro} (in the form given in \cref{prop:tf-K(^X)}, to be precise) 
is also used in the next result, whose proof combines
cohomological methods with a vanishing theorem on abelian varieties due 
to Debarre \cite{Deb95}. 

\begin{introtheorem}[cf.\ \cref{th:X-abelian}]
\label{th:X-abelian-intro}
Let $X$ be an abelian variety
and $Y\subset X$ a smooth subvariety of codimension $r$ with ample normal bundle.
Then any surjective morphism $\p \colon Y \to Z$ with $\dim Y - \dim Z > r$
extends uniquely to a morphism $\~\p \colon X \to Z$, 
and the Stein factorization of $\~\pi$ is, up to translation, a quotient morphism 
of abelian varieties $X\to X/B$ composed with a finite surjective map $X/B\to Z$.
\end{introtheorem}

The above theorem relates to a conjecture of Sommese \cite[Page 71]{Som76}
which predicts that if $X$ is a smooth complex projective variety and
$Y \subset X$ is defined by the vanishing of a regular section of an ample
vector bundle $\cE$ of rank $r$, 
then any morphism $\p \colon Y \to Z$ with $\dim Y - \dim Z > r$
extends to a morphism $\~\p \colon X \to Z$. 
\cref{th:X-abelian-intro} shows in particular that Sommese's
conjecture holds (under far less restrictive conditions on the embedding $Y \subset X$)
when $X$ is an abelian variety. 

The last part of the paper is devoted to a discussion of Sommese's conjecture. 
We look at the conjecture without assuming that $Y \subset X$ is defined by 
a regular section of an ample vector bundle on $X$, but just
requiring that $Y$ is an ample subvariety of $X$ (see \cref{conj:Som}).

\cref{th:Pic-intro} provides the first step towards the conjecture, 
implying unicity of the extension and setting up the argument for existence. 
The cohomological arguments used in the proof of \cref{th:X-abelian} yield
a general sufficient condition (a Kodaira-type vanishing) for extendability.  
Combining this with vanishing theorems due to Debarre \cite{Deb95} and Manivel \cite{Man96},
we can then easily verify the conjecture when either $Y$ or $\p$ are special, 
built up from toric and abelian varieties
(see \cref{th:special-Y,th:fiber-pi-special} for the precise statements).
We direct the reader to \cite{dFL} for further results toward the conjecture
when $\p$ has rationally connected fibers.

\subsubsection*{Acknowledgements}
We thank Daniel Litt for discussions on Sommese's conjecture 
in relation to his work \cite{Lit18}
and Adrian Langer for useful comments. 
We also thank the referee for useful comments and suggestions.
Finally, we thank Emelie Arvidsson and Lena Ji for 
pointing out an error in \cref{th:tf-intro} 
of the published version of the paper about torsion in positive characteristics, which we fix in this updated version.

\section{Ample subschemes and related properties}
\label{s:positivity}

The search for good positivity properties of arbitrary closed subschemes
of projective varieties can be traced back at least 
to \cite{Har68,Har70}, and to some extent to \cite{HM68}.

The approach followed in \cite{HM68} relies on looking at the formal neighborhood of the subscheme. 
To fix notation, let $X$ be a projective variety over a field,
and let $Y \subset X$ be a closed subscheme. 
We denote by $\^X$ the formal completion of $X$ along 
$Y$ and by $K(\^X)$ the ring of formal rational functions on $\^X$.
Note that there is a natural inclusion $K(X) \subset K(\^X)$, where $K(X)$ is the function field of $X$.
Hironaka and Matsumura looked at properties of this field extension 
as a way of encoding positivity properties of $Y$ in $X$. They introduced the following terminology.

\begin{definition}
With the above notation,  
$Y$ is said to be $G2$ in $X$ if $K(\^X)$ is a finite module over $K(X)$,
and $G3$ if the natural inclusion $K(X) \subset K(\^X)$ is an isomorphism.
\end{definition}

These conditions measure the positivity of the embedding $Y\subset X$.
For instance, 
over an algebraically closed field of characteristic zero, a locally complete intersection
subscheme with ample 
normal bundle is $G2$ \cite[Corollary~6.8]{Har68} and in fact the same is true
under a weaker positivity condition on the normal bundle (see 
\cite[Corollary~2.6]{Hal19} for more details). 
Similarly, a smooth ample subvariety (as defined below in \cref{def:ample-Y}) is $G3$ \cite[Corollary~5.6]{Ott12}.
It is clear that $G3$ implies $G2$, and an example of a subscheme $Y \subset X$ that 
is $G2$ but not $G3$ is given in \cite[Page~588]{Hir68}
(cf.\ \cite[Page~199]{Har70}).
Other related measures of posivity, introduced by Grothendieck in 
\cite{SGA2}, are the \emph{Lefschetz condition} $\Lef(X,Y)$ and the
\emph{effective Lefschetz condition} $\Leff(X,Y)$ (see \cite[Section~IV.1]{Har70} for a discussion).

The approach via formal rational functions 
was further developed in \cite{Har68,Har70,Gie77}, and later in \cite{BS96,BS02,Bad04}. 
In \cite{Har70}, Hartshorne also studied other geometric properties 
that can be interpreted as positivity conditions. 
For instance, when $X$ is smooth and $Y$ is locally complete intersection, 
Hartshorne looked at the ampleness of the normal bundle $\cN_{Y/X}$
and the condition that the cohomological dimension $\cd(X \setminus Y)$
of the complement be minimal (equal to $r-1$ where $r$ is the condimention of $Y$).
Remarkably, in characteristic zero
these two properties turned out to characterize what is now considered  the corrected notion of
ampleness (see \cref{def:ample-Y}), 
at least for locally complete intersection subvarieties of smooth projective varieties
\cite[Theorem~5.4]{Ott12}.

It took however over half a century for a general definition of ampleness for subschemes
to be formalized.
The right approach turned out to be through a suitable weakening of the notion of ampleness for line bundles, 
something that has an independent history. 
The starting point is Serre's characterization of ampleness. 
To weaken the definition of ampleness, one can require less cohomological vanishing. 
In \cite{AG62}, Andreotti and Grauert look at the eigenvalues of a metrics of a line bundle to
define the notion $q$-positivity where $q$ is a nonnegative integer, 
and relate this property to a partial cohomological vanishing. 
The first definition of a $q$-ample line bundle appears in \cite{Som78}, but Sommese's definition
is geometric and only applies to globally generated line bundles.
Partial cohomological vanishing of line bundles 
was further investigated in connection to $q$-positivity in \cite{DPS96,Dem10}, 
and an asymptotic point of view was proposed in \cite{dFKL07}.

These works eventually led to a general definition of $q$-ampleness, which was given 
by Totaro in \cite{Tot13}. 

\begin{definition}
Let $q$ be a nonnegative integer. A line bundle $\cL$ on an $n$-dimensional projective variety $X$ 
over a field is said to be, respectively:
\begin{enumerate}
\item\label{item:q-T-ample}
\emph{$q$-$T$-ample} if for a given ample line bundle $\O_X(1)$ on $X$ and for some positive integer $N$ we have
$H^{q+i}(X,\cL^{\otimes N}(-n-i)) = 0$ for $1 \le i \le n - q$;
\item\label{item:naive-q-ample}
\emph{naively $q$-ample} if for every coherent sheaf $\cF$ on $X$ we have 
$H^i(X,\cF \otimes \cL^{\otimes m}) = 0$ for all $i > q$ and all $m$ sufficiently large
depending on $\cF$; 
\item\label{item:unif-q-ample}
\emph{uniformly $q$-ample} if there exists a constant $\lambda > 0$ 
such that for every coherent sheaf $\cF$ on $X$ we have 
$H^i(X,\cL^{\otimes m}(-j)) = 0$ for all $i > q$, $j > 0$, and $m \ge \lambda j$.
\end{enumerate}
In general, we have 
$\eqref{item:unif-q-ample} \Rightarrow \eqref{item:naive-q-ample} \Rightarrow \eqref{item:q-T-ample}$, 
and the three properties are equivalent in characteristic zero \cite[Theorem~6.2]{Tot13}. 
In characteristic zero, we say that $\cL$ is \emph{$q$-ample} if it satisfies any of these equivalent 
conditions. The same terminology is used for a Cartier divisor $D$ if the condition is satisfied by 
$\O_X(D)$. 
\end{definition}

\begin{remark}
\label{r:q-ampleness-not-open-mix-char}
The notion of $q$-ampleness presents some subtleties when passing to positive characteristic.
For instance, $q$-$T$-ampleness is an open property, but naively $q$-ampleness
is not open in mixed characteristic: if it were, then
it would follow by \cite[Corollary~8.5]{Ara04} that $q$-ample line bundles on smooth varieties
of characteristic zero satisfy the Kodaira vanishing $H^i(X,\omega_X\otimes \cL) = 0$ for $i > q$, 
but there are examples where such vanishing property does not hold, see
\cite[Section~9]{Ott12} or \cite[Section~5]{Lau16}. 
\end{remark}

Using this general definition of $q$-ampleness, the notion of an ample subscheme 
was finally introduced in characteristic zero by Ottem \cite{Ott12}.

\begin{definition}
\label{def:ample-Y}
Let $X$ be a projective variety over a field of characteristic zero.
A closed subscheme $Y \subset X$ of codimension $r > 0$ is said to be 
\emph{ample} if the exceptional
divisor $E$ of the blow-up of $X$ along $Y$ is $(r-1)$-ample. 
\end{definition}

Examples are given by subschemes defined (scheme theoretically) by regular sections
of ample vector bundles \cite[Proposition~4.5]{Ott12}, 
and ample subschemes can be thought of as a generalization of this notion. 
In the smooth case, they have several similar properties such 
as having an ample normal bundle and satisfying
a Lefschetz hyperplane theorem for rational cohomology \cite[Corollary~5.6]{Ott12}, 
though some differences occur:
for instance, the Lefschetz hyperplane theorem for integral cohomology
does not hold in general (e.g., see \cite[Example~7.3]{Ott12}).

\section{Submaximal cohomological codimension of the complement}

Throughout this section, let $X$ be a projective variety over an algebraically closed field 
of arbitrary characteristic, 
and let $Y \subset X$ be a closed subscheme. 
We denote by $n$ the dimension of $X$. 
By the Lichtenbaum theorem, if $Y$ is nonempty then $\cd(X \setminus Y) \le n-1$, 
see for instance \cite[Section~III.3]{Har70} or \cite[Theorem~7.2]{Bad04}.
In this section, we study properties of the embedding $Y \subset X$
that holds when the cohomological dimension is submaximal, that is, 
when $\cd(X \setminus Y) < n-1$. 
Our focus will be on properties that relate to Picard groups and Albanese maps.
The following result is due to Speiser (see also \cite[Section~V.2]{Har70}). 

\begin{theorem}[\protect{\cite[Theorem~3]{Spe73}}]
\label{th:Speiser}
Let $X$ be an $n$-dimensional smooth projective variety over an algebraically closed field of arbitrary characteristic 
and $Y\subset X$ be a closed subscheme. Then the following are equivalent:
\begin{enumerate}
\item
\label{item:G3+div}
$Y$ is $G3$ in $X$ and meets every effective divisor on $X$. 
\item
\label{item:G2+cd}
$Y$ is $G2$ in $X$ and $\cd(X \setminus Y) < n-1$. 
\end{enumerate}
\end{theorem}

Under some additional hypothesis on $X$, Speiser also claimed that the $G2$ condition in \eqref{item:G2+cd}
is not necessary. More precisely, in \cite[Theorem~3.1]{Spe80}, he stated that 
$X$ is Gorenstein and $Y\subset X$ is positive-dimensional and is contained in the smooth locus of $X$, 
then in order to guarantee that \eqref{item:G3+div} holds it suffices to assume that $\cd(X-Y)<n-1$. 
Unfortunately, the proof given there is incorrect. 

Nonetheless, we have the following property.

\begin{proposition}
\label{prop:tf}
Let $X$ be an $n$-dimensional projective variety over an algebraically closed field
and $Y\subset X$ a positive-dimensional closed subscheme. 
Assume that $\cd(X \setminus Y) < n-1$.
Then $K(\^X)$ is a field and $K(X)$ is algebraically closed in $K(\^X)$. 
\end{proposition}

\begin{proof}
By \cite[Theorem~7.6]{Bad04}, the condition on cohomological dimension
implies that $Y$ is connected. 
In order to prove that $K(\^X)$ is a field and 
$K(X)$ is algebraically closed in $K(\^X)$, 
we may assume without loss of generality that $X$ is normal.
Indeed, if $\pi\colon W\to X$ is the normalization map 
then we see by \cite[Proposition~1.1]{Har68} that 
$\cd(W\setminus \pi^{-1}(Y))= \cd(X\setminus Y)\leq n-2$, so the hypothesis holds on 
the normalization, and if $\^W$ is the formal completion along $\pi^{-1}(Y)$ then 
the canonical map $K(\^X) \to K(\^W)$ is an isomorphism by \cite[Theorem~2.2]{BS02}.
Thus the claimed property can be deduced from the normalization. 

By \cite[Lemma~(1.4)]{HM68}, the connectedness of $Y$ implies that $K(\^X)$ is a field. 
Suppose by contradiction that $K(X)$ is not algebraically closed in $K(\^X)$. 
Then there is an element $\eta \in K(\^X)\setminus K(X)$ that is algebraic over $K(X)$. 
As we are now assuming that $X$ is normal, \cite[Theorem~2.5]{BS02}
implies that there is a finite surjective morphism $f\colon X'\to X$ from a variety $X'$
and a closed subscheme $Y'\subset X'$ such that 
$f$ is \'etale at each point of $Y'$ and its restriction to $Y'$ gives an isomorphism $Y' \to Y$; 
moreover, if $\^{X'}$ is the formal neighborhood of $Y'$ in $X'$ and $\^f \colon \^{X'} \to \^X$
is the morphism of formal schemes induced by $f$, then $\^f^*\eta\in K(X')$.
Note that $X'$ is projective, since $X$ is. 
By Stein factorization and the fact that we are assuming that $X$ is normal, we have $\deg(f)>1$.
The induced map $(X'\setminus f^{-1}(Y))\to (X\setminus Y)$ is finite and surjective, and
therefore $\cd(X'\setminus f^{-1}(Y))= \cd(X\setminus Y)\leq n-2$ by \cite[Proposition~1.1]{Har68}.
Thus $f^{-1}(Y)$ has to be connected, again by \cite[Theorem~7.6]{Bad04}.
As $f^{-1}(Y)$ contains $Y'$ and $f$ is \'etale at all points of $Y'$, it follows that $Y'$ 
must be a connected component in $f^{-1}(Y)$, but this contradicts the connectedness of $f^{-1}(Y)$ since $f$
has degree $> 1$ and its restriction to $Y'$ has degree 1. 
\end{proof}

The condition that $K(\^X)$ is a field and $K(X)$ is algebraically closed in $K(\^X)$ was studied 
by B\u adescu and Schneider in \cite{BS02}, where they prove that it is 
equivalent to the geometric condition saying 
that for any proper surjective morphism of projective varieties $g\colon V\to X$, the set $g^{-1}(Y)$ is connected.
This condition can also be regarded as an orthogonal condition to $G2$, 
in the sense that the two conditions together are equivalent to $G3$.

Our next result relates the condition that $K(\^X)$ is a field and $K(X)$ is algebraically closed in $K(\^X)$
to properties of the restriction map on Picard groups.

\begin{proposition}
\label{prop:tf-K(^X)}
Let $X$ be a projective variety over an algebraically closed field $k$ of arbitrary characteristic, 
and let $Y\subset X$ be a positive-dimensional closed subscheme. 
Assume that $K(\^X)$ is a field and $K(X)$ is algebraically closed in $K(\^X)$ (e.g., $Y$ is $G3$ in $X$).
Then:
\begin{enumerate}
\item
\label{item:Pic-ker-tf}
The kernel of the map $\Pic(X)\to\Pic(Y)$ is torsion free if $\charK k = 0$, 
and its torsion is killed by a power of $p$ if $\charK k = p > 0$. 
\item
\label{item:Pic0-inj}
The map $\Pic^0(X)\to\Pic^0(Y)$ is injective if $\charK k = 0$, 
and has finite kernel of order a power of $p$ if $\charK k = p > 0$.
\item
\label{item:Alb-onto}
The map $\Alb(Y) \to \Alb(X)$ is surjective, and has connected fibers if $\charK k = 0$. 
\end{enumerate}
\end{proposition}

\begin{proof}
To prove \eqref{item:Pic-ker-tf}, we argue by contradiction
and suppose there is a non-trivial, torsion line bundle $\mathcal{L}\in \Pic(X)$ such that
$\mathcal{L}|_Y\cong \O_Y$. Let $m$ be the least positive integer such that $\mathcal{L}^{\otimes m}\cong \O_X$
and assume that $m$ is coprime to $p$ if $\charK k = p > 0$.
We may take the \'etale cyclic Galois cover $\pi\colon \~X\to X$ associated with 
$\mathcal{L}$ and the section $1\in\Gamma(\O_X)\cong\Gamma(\mathcal{L}^{\otimes m})$.
Since 
\[
H^{0}(\~X,\O_{\~X})\cong H^{0}(X,\pi_*\O_{\~X})
\cong H^{0}(X,\O_X\oplus \mathcal{L}^{-1}\oplus\cdots\oplus\mathcal{L}^{-m+1})\cong H^{0}(X,\O_X),
\]
$\~X$ is connected.
Therefore, $\~X$ is irreducible and reduced.
By contrast, the fact that $\mathcal{L}|_Y\cong \O_Y$ implies that $\pi^{-1}(Y)$ consists of 
exactly $m$-copies of $Y$. As $m \ge 2$, this contradicts \cite[Theorem~2.7]{BS02}
which implies that $\p^{-1}(Y)$ must be connected. 

Since a positive dimensional kernel of a morphism
between abelian varieties contains $m$-torsion for every positive integer $m$, 
\eqref{item:Pic0-inj} follows from \eqref{item:Pic-ker-tf}. 

To finish the proof, we deduce \eqref{item:Alb-onto} from \eqref{item:Pic0-inj}. 
Denote for short $A := \Pic^0(Y)$ and $B := \Pic^0(X)$, 
so that $A^\vee = \Alb(Y)$ and $B^\vee = \Alb(X)$. 
Let $G$ denote the kernel of $B \to A$, and denote $\ov B = B/G$. We factor the map $B \to A$ as
the composition of the isogeny $B \to \ov B$ and the injective homomorphism $\ov B \to A$. 
We obtain the decomposition of the dual map $A^\vee \to B^\vee$ 
as the composition of the homomorphisms $A^\vee \to \ov B^\vee$ and 
$\ov B^\vee \to B^\vee$. 
By the dual exact sequence, $\ov B^\vee \to B^\vee$ is a isogeny, 
and Poincar\'e's reducibility theorem implies that $A^\vee \to \ov B^\vee$ is surjective. 
To see this last property, we apply
Poincar\'e's reducibility theorem to find an abelian subvariety $C \subset A$
such that the map $\ov B \times C \to A$ given by $(\ov b,c) \mapsto \ov b+c$ is an isogeny;
dualizing, we obtain an isogeny $A^\vee \to (\ov B \times C)^\vee$, and since
the inclusion $\ov B \subset \ov B \times C$ gives a surjection $(\ov B \times C)^\vee \to \ov B^\vee$, 
the resulting map $A^\vee \to \ov B^\vee$ is surjective. 
Therefore, the dual map $A^\vee \to B^\vee$ is surjective.

Regarding the last assertion in \eqref{item:Alb-onto}, assume that $\charK k = 0$. 
Note that in this case $G = \{0\}$. 
Let $K$ be the kernel of $A^\vee \to B^\vee$, and let $K_0 \subset K$ be the connected component through the origin.
If $A^\vee \to B^\vee$ is not a morphism with connected fibers, then $K_0 \ne K$, and hence the map
factors as $A^\vee \to A^\vee/K_0 \to B^\vee$ where the last arrow is an isogeny
of degree $> 1$. Dualizing, this contradicts the fact that the map $B \to A$ is injective. 
\end{proof}

We will use the following corollary, which is a direct consequence of
\cref{prop:tf,prop:tf-K(^X)}.

\begin{corollary}
\label{cor:tf}
Let $X$ be an $n$-dimensional projective variety over an algebraically closed field $k$
and $Y\subset X$ a positive-dimensional closed subscheme. 
Assume that $\cd(X \setminus Y) < n-1$.
Then:
\begin{enumerate}
\item
\label{item:Pic-ker-tf}
The kernel of the map $\Pic(X)\to\Pic(Y)$ is torsion free if $\charK k = 0$, 
and its torsion is killed by a power of $p$ if $\charK k = p > 0$. 
\item
\label{item:Pic0-inj}
The map $\Pic^0(X)\to\Pic^0(Y)$ is injective if $\charK k = 0$, 
and has finite kernel of order a power of $p$ if $\charK k = p > 0$.
\item
\label{item:Alb-onto}
The map $\Alb(Y) \to \Alb(X)$ is surjective, and has connected fibers if $\charK k = 0$.
\end{enumerate}
\end{corollary}

We finish this section with two applications of \cref{prop:tf-K(^X)} 
which are of independent interest.

A theorem attributed by Hartshorne to Matsumura (cf.\ \cite[Exercise~III.4.15]{Har70})
states that if $X$ is a positive-dimensional smooth proper variety
and $Y \subset X$ is a smooth closed subvariety with ample normal bundle then 
the induced map $\Alb(Y) \to \Alb(X)$ is surjective. 
As we already mentioned, over an algebraically closed field of characteristic zero, 
a smooth subvariety with ample normal bundle is $G2$.
The following result can be seen as a strengthened version of Matsumura's result.

\begin{corollary}
	\label{cor:G2}
Let $X$ be a normal projective variety over an algebraically closed field of arbitrary characteristic, 
and let $Y\subset X$ be a positive-dimensional connected closed subscheme. 
Assume that $Y$ is $G2$ in $X$.
Then the induced map $\Alb(Y)\to \Alb(X)$ is surjective.	
\end{corollary}

\begin{proof}
By the proof of \cite[Theorem~4.3]{Gie77}, there is a finite surjective map $f\colon X'\to X$ 
and a closed subscheme $Y'\subset X'$, such that $Y'$ is $G3$ in $X'$, 
$f$ is \'etale at each point of $Y'$, and the restriction of $f$ to $Y'$ gives an isomorphism $Y' \to Y$. 
By \cref{prop:tf-K(^X)}, the map $\Alb(Y')\to \Alb(X')$ is surjective.
Since $f$ is surjective, $\Alb(X')\to \Alb(X)$ is surjective, and therefore 
$\Alb(Y)\to \Alb(X)$ is also surjective.
\end{proof}	

It is well known that, with the sole exception of $\P^n$, the diagonal never has ample normal 
bundle inside the self-product $X\times X$ of a 
smooth projective variety $X$, and it follows by Ottem's Lefschetz hyperplane theorem 
\cite[Corollary~5.3]{Ott12}
that in fact $\P^1$ is the only variety for which the diagonal is an ample subvariety of $X \times X$. 
In \cite[Theorem~2]{BS96}, B\u{a}descu and Schneider proved that for any rational homogeneous space $X$ over $\C$,
the diagonal $\Delta_X$ is $G3$ in $X\times X$.	
In \cite[Proposition~3.10]{Hal19}, Halic showed that 
for a smooth projective variety $X$ with non-pseudoeffective cotangent bundle over an algebraically closed field,
the diagonal is $G2$. In the opposite direction, we obtain the following property.

\begin{corollary}
\label{cor:diagonal}
	Let $X$ be a projective variety with nontrivial Albanese variety, and 
	let $\^{X\times X}$ denote the completion of $X \times X$ along the diagonal $\Delta_X$. 
	Then $K(X\times X)$ is not algebraically closed in $K(\widehat{X\times X})$ and $\Delta_X$ is not $G2$ in $X\times X$.
\end{corollary}

\begin{proof}
	As $\Alb(X\times X)\cong \Alb(X)\times\Alb(X)$, 
	the induced map $\Alb(\Delta_X)\to \Alb(X\times X)$ cannot be surjective,
	hence the \lcnamecref{cor:diagonal} follows from \cref{prop:tf-K(^X)} and \cref{cor:G2}.
\end{proof}

\section{A Grothendieck--Lefschetz theorem}
\label{s:Pic}

We now focus on the case of characteristic zero. 
The purpose of this section is to establish a Grothendieck--Lefschetz type theorem 
in the context of ample subvarieties. 

When $Y \subset X$ is an ample divisor, one 
can deduce the Grothendieck--Lefschetz theorem directly from 
the Lefschetz hyperplane theorem with integral coefficients and Kodaira vanishing 
(e.g., see \cite[Example~3.1.25]{Laz04}), 
and the case where the subvariety is defined by the vanishing 
of a regular section of an ample vector bundle is also well understood thanks
to the Lefschetz--Sommese theorem \cite[Lemma~A]{Som76} (see also \cite[Section~7.1.A]{Laz04}). 
We should remark that Grothendieck's proof \cite[Corollary~XII.3.6]{SGA2}, 
which applies to the case where $Y$ is any effective ample divisor,
used different methods that do not require $Y$ to be reduced
(see \cite[Chapter~IV]{Har70} and \cite[Remark~3.1.26]{Laz04} for a discussion).
However, Grothendieck's proof relies on the fact that, being an ample divisor, 
$Y$ satisfies the effective Lefschetz condition $\Leff(X,Y)$, 
which is no longer the case in general if $Y$ is an ample subvariety of codimension $\ge 2$
(see \cref{eg:veronese} below). 
Note also that we cannot rely on Kodaira vanishing when $Y$ has higher codimension
(cf.\ \cref{r:q-ampleness-not-open-mix-char}).

To deal with the case where $Y$ is an ample subvariety of arbitrary codimension, 
we will use the Lefschetz hyperplane theorem 
with rational coefficients from \cite{Ott12} in combination with \cref{cor:tf}, and 
assume that $Y$ is smooth so that we can rely on Hodge theory in place of Kodaira vanishing.

We start by recalling Ottem's result.

\begin{theorem}[\protect{\cite[Section~5]{Ott12}}]
\label{th:Lefschetz}
Let $X$ be an $n$-dimensional smooth complex projective variety, 
and let $Y \subset X$ be a smooth subvariety such that $\cd(X\setminus Y)= q$. Then the map
$H^i(X,\Q)\to H^i(Y,\Q)$ is an isomorphism for $i < n -1-q$ and is injective for $i = n-1-q$.
\end{theorem}

As in the classical setting, one deduces restriction properties at the
level of Picard groups. Because of the rational coefficients, 
torsion is possible, and we will use \cref{cor:tf} to control
the torsion in the kernel. Note, by contrast, that as shown by simple examples
(see \cref{eg:veronese}), torsion can appear in the cokernel.

\begin{theorem}
\label{l:Pic-iso}
Let $X$ be an $n$-dimensional 
smooth projective variety defined over an algebraically closed field of characteristic zero, 
and let $Y \subset X$ be a smooth closed subvariety. 
Assume that $\cd(X\setminus Y)\leq n-3$
(e.g., $Y$ is an ample subvariety of dimension $\ge 2$). Then:
\begin{enumerate}
\item\label{Pic0} 
The canonical map $\Pic^{0}(X)\to\Pic^{0}(Y)$ is an isomorphism. 
\item\label{NS} 
The restriction map $\NS(X)\to \NS(Y)$ is injective.
\item\label{Pic} 
The restriction map $\Pic(X)\to \Pic(Y)$ is injective.
\end{enumerate}
Moreover, if $\cd(X\setminus Y)\leq n-4$ (e.g., $Y$ is an ample subvariety of dimension $\ge 3$), then
the maps in \eqref{NS} and \eqref{Pic} have finite cokernel.
\end{theorem}

\begin{proof}
We can reduce to the case where the ground field is the field of complex numbers, as follows. 
First, we find a finitely generated field extension $k_0$ of $\Q$ such that both $X$ and $Y$ are defined over $k_0$, 
and embed $k_0$ into $\C$. Then, we can find an algebraically closed field $K$
that contains both the original ground field and $\C$. 
Letting $k$ denote either the original ground field or $\C$, 
and considering $X$ and $Y$ as being defined over $k$, 
we only need to check that both hypothesis and conclusions of the theorem hold true
over $k$ if and only if they hold true over $K$. 
Regarding the hypothesis, we use \cite[Proposition~3.1]{Har70}
and the fact that for an ample line bundle $\O(1)$ on $X$ we have that
$\O(1)_K$ is ample and $H^i(X_K \setminus Y_K, \O(m)_K) = H^i(X \setminus Y, \O(m)) \otimes_kK$
for all $i$ and $m$,
to check that $\cd(X_K\setminus Y_K) = \cd(X \setminus Y)$. 
We denote $\PicS(X)$ and $\PicV(X)$ to be the Picard scheme and the Picard variety of $X$ respectively.
From the fact that $\PicS(X_K) \cong \PicS(X)_K$ and $\PicV(X_K) \cong \PicV(X)_K$,
we see that \eqref{Pic0} and \eqref{Pic} hold over $k$ if and only if they hold over $K$, 
and this shows that the same equivalence holds for \eqref{NS}. 
As for the last assertion, it suffices to look at the cokernel
of $\NS(X) \to \NS(Y)$.
Since $k$ is algebraically closed, 
the irreducible components of $\PicS(X)$ correspond to that of $\PicS(X_K)$ via base change.
This correspondence gives a natural identification $\NS(X) = \NS(X_K)$
and similarly we have $\NS(Y) = \NS(Y_K)$. 
These identifications are compatible with the restriction maps, 
and therefore they induce an identification between their cokernels. 

Consider the compatible exponential sequences on
the associated complex analytic varieties $X^\an$ and $Y^\an$: 
\begin{equation}
\label{eq:exp}
\xymatrix{
0 \ar[r] & \Z_{X^{\an}} \ar[r]\ar[d] & \O_{X^\an} \ar[d]\ar[r]^\exp & \O_{X^\an}^* \ar[r]\ar[d] & 0 \\
0 \ar[r] & \Z_{Y^{\an}} \ar[r] & \O_{Y^\an} \ar[r]^\exp & \O_{Y^\an}^* \ar[r] & 0
}.
\end{equation}
We shall apply the GAGA principle liberally in the following
and, as usual, drop the superscript in the notation 
of complex analytic varieties when writing singular cohomology. 
If $\cd(X\setminus Y)\leq n-4$, then \cref{th:Lefschetz} implies that for $i=1,2$ the map 
$H^i(X,\Z) \to H^i(Y,\Z)$ has finite kernel and cokernel, 
and by Hodge theory this implies that 
$H^i(X,\O_X) \to H^i(Y,\O_Y)$ is an isomorphism for $i=1,2$.
If $\cd(X\setminus Y)\leq n-3$, then these conclusions hold for $i=1$. 

For \eqref{Pic0}, we take the long exact sequence in cohomology associated 
with $\eqref{eq:exp}$.
Observing that $H^{0}(X,\O_X)\to H^{0}(X,\O^*_X)$ is nothing but the exponential function
$\C\to \C^*$ and is surjective, we have the diagram with exact rows
\[
\xymatrix{
0 \ar[r] & H^{1}(X,\Z) \ar[r]\ar[d] & H^{1}(X,\O_{X}) \ar[d]^{\cong}\ar[r] & \Pic^{0}(X) \ar[r]\ar[d] & 0 \\
0 \ar[r] & H^{1}(Y,\Z)  \ar[r] & H^{1}(Y,\O_{Y})\ar[r] & \Pic^{0}(Y) \ar[r] & 0
}.
\]
By inverting the middle vertical arrow, we see that $\Pic^0(X)\to \Pic^0(Y)$ is surjective, and 
therefore it is an isomorphism by \cref{cor:tf}.

For \eqref{NS} and \eqref{Pic}, note that we have  commutative diagrams
with exact rows
\[
\xymatrix{
0 \ar[r] & \NS(X) \ar[r]\ar[d] & H^{2}(X,\Z) \ar[d]\ar[r] & C_1 \ar[r]\ar[d] & 0 \\
0 \ar[r] & \NS(Y)  \ar[r] & H^{2}(Y,\Z)\ar[r] & C_2 \ar[r] & 0
}
\]
and
\[
\xymatrix{
0 \ar[r] & C_1 \ar[r]\ar[d] & H^{2}(X,\O_X) \ar[d]^{\cong}\ar[r] & D_1 \ar[r]\ar[d] & 0 \\
0 \ar[r] & C_2  \ar[r] & H^{2}(Y,\O_{Y})\ar[r] & D_2 \ar[r] & 0
}.
\]
The fact that the middle vertical arrow in the second diagram is an isomorphism follows from 
the hypothesis that $\cd(X\setminus Y)\leq n-3$. 
We see from the second diagram that $C_1\to C_2$ is injective.
Going back to the first diagram, we see that 
$\Ker(\NS(X)\to\NS(Y))\cong \Ker(H^{2}(X,\mathbb{Z})\to H^{2}(Y,\mathbb{Z}))$, which is finite.
Moreover, if $\cd(X\setminus Y)\leq n-4$, then we also have that
$\Coker(\NS(X)\to \NS(Y))\subset \Coker(H^{2}(X,\mathbb{Z})\to H^{2}(Y,\mathbb{Z}))$, which is finite.

We see by the diagram
\[
\xymatrix{
0 \ar[r] & \Pic^0(X) \ar[r]\ar[d] & \Pic(X)\ar[r]\ar[d] & \NS(X) \ar[r]\ar[d] & 0 \\
0 \ar[r] & \Pic^0(Y)  \ar[r] & \Pic(Y)\ar[r] & \NS(Y) \ar[r] & 0
}
\]
that $\Ker(\Pic(X)\to\Pic(Y))\cong \Ker(\NS(X)\to\NS(Y))$ and $\Coker(\Pic(X)\to\Pic(Y))\subset \Coker(\NS(X)\to\NS(Y))$, hence both kernel and cokernel of $\Pic(X)\to\Pic(Y)$ are finite.
We then apply again \cref{cor:tf} to conclude that the kernel is trivial.
\end{proof}

The next example shows that the conclusions in \cref{l:Pic-iso} are optimal. 

\begin{example}
\label{eg:veronese}
With the same assumptions as in \cref{l:Pic-iso},
the map $\Pic(X)\to \Pic(Y)$ can have nontrivial torsion in the cokernel.
For example, consider the $d$-uple embedding $\mathbb{P}^n\hookrightarrow\mathbb{P}^N$
over the complex numbers.
Since in this case the Lefschetz hyperplane theorem with rational coefficient holds, 
it follows by a theorem of Ogus that
$\mathbb{P}^n$ is an ample subvariety of $\mathbb{P}^N$
(cf.\ \cite[Theorem~7.1]{Ott12}).
Nonetheless, the induced map on Picard groups has cokernel isomorphic to $\mathbb{Z}/d\mathbb{Z}$.
Note that this also gives an example where the 
effective Lefschetz condition $\Leff(X,Y)$ 
fails for an ample subvariety. By contrast, it follows by \cite[Proposition~IV.1.1]{Har70}
that the Lefschetz condition $\Lef(X,Y)$
always holds for ample locally complete intersection subschemes of smooth projective varieties. 
\end{example}

We apply \cref{l:Pic-iso} to compare the Albanese map of a variety to that of an ample subvariety.

\begin{corollary}
\label{th:Alb}
Let $X$ be a smooth variety defined over an algebraically closed field of characteristic zero 
and $Y \subset X$ a smooth ample subvariety of dimension $\ge 2$.
\begin{enumerate}
\item
\label{item:Alb(i)=isom}
The induced map on Albanese varieties $\Alb(Y)\to \Alb(X)$ is an isomorphism, 
hence we have the commutative diagram
\[
\xymatrix{
Y \ar@{^(->}[r] \ar[d]_{\alb_Y} &X \ar[d]^{\alb_X}\\
\Alb(Y) \ar[r]^\cong &\Alb(X)
}.
\]
\item
\label{item:alb=alg-fiber-sp}
If $\alb_Y$ is an algebraic fiber space, then so is $\alb_X$. 
\item
\label{item=not-mAd}
If $Y$ is not of maximal Albanese dimension, then neither is $X$, and furthermore
$\alb_X$ and $\alb_Y$ have the same image and share the same Stein factorization.
\end{enumerate}
\end{corollary}

\begin{proof}
By \cref{l:Pic-iso}, $\Pic^0(X) \to \Pic^0(Y)$ is an isomorphism, hence
so is the dual map $\Alb(Y) \to \Alb(X)$. This gives \eqref{item:Alb(i)=isom}.
\eqref{item:alb=alg-fiber-sp} and \eqref{item=not-mAd} follow from
\cref{l:afs} and \cref{l:Stein} respectively.
\end{proof}

\begin{lemma}\label{l:afs}
Let $X$ be a projective variety defined over a field $k$ of arbitrary characteristic and $Y\subset X$ be a subvariety.
Let $\pi\colon X\to Z$ be a morphism to a projective variety $Z$ such that its restriction to $Y$, 
$\pi|_Y \colon Y\to Z$, is an algebraic fiber space.
Then $\pi$ is also an algebraic fiber space.
\end{lemma}

\begin{proof}
Let $X\to \BroadSpec \p_*\O_X\to Z$ be the Stein factorization of $\p$.
The composition $\O_Z\to \p_*\O_X\to (\p|_Y)_*\O_Y$ is an isomorphism, since $\pi|_Y$ is an algebraic fiber space.
This implies $\O_Z\to \p_*\O_X$ split-injects,
giving a section of $\BroadSpec \p_*\O_X\to Z$.
Since $X$ is connected, $\BroadSpec \p_*\O_X$ is also connected.
Thus, $\O_Z\to \p_*\O_X$ is an isomorphism.
\end{proof}

\begin{lemma}\label{l:Stein}
Let $X$ be a smooth projective variety defined over an algebraically closed field of characteristic zero, 
and let $Y\subset X$ be a smooth ample subvariety.
Let $\pi\colon X\to Z$ be a dominant morphism to a projective variety $Z$ 
such that the general fiber dimension of $\pi|_Y$ is positive.
Then $\pi$ and $\pi|_Y$ have the same Stein factorization.
\end{lemma}

\begin{proof}
By \cite[Theorem B]{Lau18}, $\pi|_Y:Y\to Z$ is dominant as well.
Let $Y\to Z'\to Z$ be the Stein factorization of $\pi|_Y$ and $X\to W\to Z$ be the Stein factorization of $\pi$.
Then the morphism $Z'\to Z$ factors through $W\to Z$ and the induced map $g\colon Z'\to W$ is finite surjective.
We claim that $g$ is an isomorphism.
To see this, fix a general closed point $w \in W$.
Note that $X_w$ is a smooth connected variety and $X_w \cap Y$ is positive-dimensional.
In fact, $X_w\cap Y$ is the union of the fibers of $Y\to Z'$ over 
$g^{-1}(w)$, which is a finite set.
In particular, $X_w\cap Y$ is connected if and only if $g$ is birational.
By \cite[Proposition~4.8]{Lau16}, $X_w \cap Y$ is an ample subscheme in $X_w$.
Applying \cref{th:Lefschetz}, we see that $X_w \cap Y$ must be connected, 
hence $g$ is birational. Since $W$ is normal, it follows that $g$ is an isomorphism.
This completes the proof.
\end{proof}

\section{Applications to abelian varieties}

The results of this section hold over an algebraically closed field of characteristic zero.

The Grothendieck--Lefschetz theorem was used by Sommese in \cite{Som78} to study properties of 
abelian varieties in the context of manifolds with special ample divisors. 
One of the results he proved is that abelian varieties of dimension $\ge 2$
cannot be realized as ample divisors in smooth projective varieties
(see the first assertion in \cite[Corollary~I-A]{Som76}).
The condition of the dimension of the abelian variety is of course 
necessary, since elliptic curves can be embedded as cubics in $\P^2$. 

As a direct application of \cref{l:Pic-iso}, we generalize this result
of Sommese to higher codimensions, in the following sense.

\begin{proposition}
	\label{th:Y=A-cannot-be-ample}
	Let $A$ be an abelian variety.
		If $\dim A\geq 2$, 
		then $A$ cannot be realized as an ample subvariety of any smooth projective variety.
\end{proposition}

\begin{proof}
	Suppose by contradiction that $A$ is an ample subvariety of a smooth projective variety $X$.
	By \cref{th:Alb}, the Albanese morphism on $X$ gives a retraction from $X$ to $A$.
	Since the map $\Pic(X)\to \Pic(A)$, induced by the inclusion $A\subset X$, is injective by \cref{l:Pic-iso},
	and the composition $\Pic(A)\to \Pic(X)\to \Pic(A)$ is the identity map,
	the map $\Pic(A)\to \Pic(X)$ induced by the retraction map is an isomorphism.
	This implies, by \cref{l:finite}, that the retraction map is finite, which is impossible.
\end{proof}

\begin{lemma}
	\label{l:finite}
	Let $f\colon X\to Z$ be a morphism a projective varieties.
	If $f^*\colon \Pic(Z)\to \Pic(X)$ has torsion cokernel, then $f$ is a finite morphism.
\end{lemma}

\begin{proof}
	Take an ample line bundle $\cL$ on $X$.
	By the hypothesis, after replacing $\cL$ by a suitable multiple, $\cL$ can be expressed as the pull-back of a line bundle $\mathcal{A}$ on $Z$.
	If there is any irreducible curve $C$ contracted by $f$, then $\cL\cdot C=f^*\mathcal{A}\cdot C=0$, contradicting the assumption that $\cL$ is ample.
\end{proof}

The second assertion of Corollary I-A 
	as well as Propositon I and Corollary I-B of \cite{Som76} can also be extended in a similar fashion
	from ample divisors 
	to ample subvarieties of higher codimension by combining 
	the same arguments given there with \cref{th:Alb}. 
More results which fit within the context of \cite{Som76} will be proved in the next section
(see \cref{c:Y=AxB,c:abelian-fibrations}).	

As we shall further discuss later, 
the next result also relates closely to \cite{Som76}.

\begin{theorem}
\label{th:X-abelian}
Let $X$ be an abelian variety
and $Y\subset X$ a smooth subvariety of codimension $r$ with ample normal bundle.
Then any surjective morphism $\p \colon Y \to Z$ with $\dim Y - \dim Z > r$
extends uniquely to a morphism $\~\p \colon X \to Z$. 
Moreover, the Stein factorization of $\~\pi$ is, up to translation, a quotient morphism 
of abelian varieties $X\to X/B$ composed with a finite surjective map $X/B\to Z$.
\end{theorem}

\begin{proof}
We may reduce to the case where the ground field is $\C$. 
To see this, assume that $X$ and $Y$ are defined over an 
algebraically closed field $k$, 
and let $K$ be an algebraically closed field extension. Since 
$\cN_{Y_K/X_K} = (\cN_{Y/X})_K$, we see that
$\cN_{Y/X}$ is ample if and only if $\cN_{Y_K/X_K}$ is ample. 
Let $\Hom_k(X,Z)$ be the scheme whose closed points parametrize $k$-morphisms from $X\to Z$.
Note that $\p$ corresponds to a closed point on $\Hom_k(Y,Z)$.
Let $\Hom_k(X,Z;\p)$ be the fiber of the natural restriction morphism $\Hom_k(X,Z)\to \Hom_k(Y,Z)$ over this point.
We have $\Hom_k(X,Z,\pi)\times_{\Spec k}\Spec K\cong \Hom_K(X_K,Z_K,\pi_K)$, where $\pi_K\colon Y_K\to Z_K$ is the base change of $\p$ over $K$.
Thus, $\Hom_k(X,Z,\pi)$ is nonempty if and only if $\Hom_K(X_K,Z_K,\pi_K)$ is nonempty.
Note also that the proof that an extension $\~\p$ of $\p$ satisfies the properties listed in the 
last part of the statement uses \cref{l:mapfromA}, which does not require working over the
complex numbers. 
Therefore both hypothesis and conclusions hold over $k$ if and only if they hold over $K$, 
and this allows us to reduce to the case where the ground field is $\C$.

We may assume that $\dim Z\geq 1$ as otherwise there is nothing to prove, hence $\dim Y\geq r+2$.
By the Barth--Lefschetz theorem on abelian varieties \cite[Theorem~4.5]{Deb95},
$H^q(X,Y;\mathbb{C})=0$ for $q\leq \dim Y-r+1$.
Since $\dim Y-r+1\geq 3$, 
it follows from the proof of \cref{l:Pic-iso} that $\Pic(X)\to\Pic(Y)$ has finite kernel and cokernel.
In fact, since $2\dim Y\geq \dim X$, we see  
by \cite[Proposition~4.8]{BS02} that $Y$ is $G3$ in $X$, that is, $K(X)\cong K(\^X)$ where $\^X$ is the formal completion along $Y$. Therefore $\Pic(X)\to\Pic(Y)$ is injective by \cref{prop:tf-K(^X)}.
Thus, if $\mathcal{A}$ is a sufficiently divisible very ample line bundle on $Z$, 
then there exists a line bundle $\mathcal{L}$ on $X$ such that $\mathcal{L}|_Y\cong\pi^*\mathcal{A}$.

By \cite[Theorem 4.4]{Deb95}, for any nef line bundle $\cM$ on $Y$ we have
\[
H^{i}(Y,\omega_Y\otimes \Sym^{t}\mathcal{N}_{Y/X}\otimes \mathcal{M})=0
\quad\text{for $i\ge r$, $t\ge 1$.}
\]
On the one hand, since $Y$ has ample normal bundle, it is geometrically non-degenerate in $X$, and therefore 
for any subvariety $W\subset X$ with $\dim W +\dim Y\geq \dim X$ we have $W\cap Y\neq \emptyset$, see
\cite[Corollary~2.5]{Deb}.
In particular, the complement $X\setminus Y$ supports no divisors, 
and since, as we already remarked, $Y$ is G3 in $X$, it follows by 
\cref{th:Speiser} that $\cd(X\setminus Y)\leq \dim X - 2$.
Therefore \cref{l:H^0(L)=H^0(L|Y)}
implies that the restriction map 
\begin{equation}
\label{eq:bij}
H^0(X,\cL) \to H^0(Y,\cL|_Y)
\end{equation}
is bijective.

Now take $(\dim Z+1)$-many general sections in $H^0(X,\cL)$.
They cut out a closed subscheme $B'\subset X$, which is either empty or has dimension 
$\ge \dim X-\dim Z - 1$.
By what we observed before about $Y$
meeting every subvariety $W \subset X$ of complementary dimension, 
if $B'$ is nonempty then it must intersect $Y$, 
which contradicts the fact that \eqref{eq:bij} is surjective
and $\cL|_Y$ is globally generated. Thus, $|H^0(X,\cL)|$ is base-point free.

We have the commutative diagram
\[
\xymatrix{
Y \ar[r]^\pi \ar@{^(->}[d] & Z \ar@{^(->}[r] & Z' \ar@{^(->}[r] & P := \mathbb{P}(H^{0}(Z,\cA))\\
X \ar[rru]_{\~\pi}
}
\]
where $Z'$ is the image of $\~\pi$.
If $Z' = Z$, then we are done.
Suppose otherwise that $Z \subsetneq Z'$.
After replacing $\cA$ in the beginning with a sufficiently large multiple, we may assume 
that $\O_{Z'}\cong\~\pi_*\O_X$ \cite[Theorem~2.1.27]{Laz04} and that 
$H^{0}(Z',\cI_Z\otimes\O_P(1)|_{Z'})\neq 0$.
Now, $H^{0}(Z',\O_P(1)|_{Z'})$ contains a section that vanishes on $Z$.
Since $H^{0}(X,\cL) \cong H^{0}(Z',\O_P(1)|_{Z'})$, 
this means that $H^{0}(X,\cL)$
contains a section on $X$ that vanishes on $Y$,
contradicting the bijectivity in \eqref{eq:bij}.

Unicity of the extension $\~\p$ follows from \cref{prop:unique}. 
The last statement of the \lcnamecref{th:X-abelian} follows from \cref{l:mapfromA}.
\end{proof}

\begin{lemma}
\label{l:H^0(L)=H^0(L|Y)}
Let $X$ be an $n$-dimensional smooth projective variety and $Y \subset X$ a smooth subvariety
equipped with a morphism $\p \colon Y \to Z$.
Assume that $\cd(X \setminus Y) \le n - 2$
and that for any globally generated line bundle $\cM$ on $Y$
\begin{equation}
\label{eq:Kodaira-type}
H^{i}(Y,\omega_Y\otimes \Sym^{t}\mathcal{N}_{Y/X}\otimes \mathcal{M})=0
\quad\text{for $i\ge r$ and $t\ge 1$.}
\end{equation}
Then for any line bundle $\cL$ on $X$ such that $\cL|_Y$ is the pull-back of a very ample line bundle $\cA$
on $Z$, the restriction map $H^0(X,\cL) \to H^0(Y,\cL|_Y)$ is a bijection. 
\end{lemma}

\begin{proof}
Let $\s \colon \~X \to X$ be the blow-up of
$X$ along $Y$, with exceptional divisor $E$. 
By definition of ampleness of $Y$ in $X$, $E$ is an $(r-1)$-ample divisor on $\~X$. 
Note that $E \cong \P(\cN_{Y/X}^*)$. Let $\~\cL := \s^*\cL$ be the pull-back of $\cL$.
We claim that
\begin{equation} 
\label{eq:desired-vanishing}
H^{n-1-k}\bigl(E,\omega_E(tE) \otimes \~\cL^*\bigr)=0 
\quad \text{for $t \ge 1$ and $k=0,1$.}
\end{equation}
Note that these cohomology groups are isomorphic to
$H^{n-r-k}(Y,\omega_Y \otimes \Sym^t\cN_{Y/X} \otimes \cL^*)$.
Here, we implicitly used the fact that working over a field of characteristic zero, 
we have $(\Sym^t\cN_{Y/X}^*)^* \cong \Sym^t\cN_{Y/X}$.

In order to prove \eqref{eq:desired-vanishing}, first observe that our hypothesis 
\eqref{eq:Kodaira-type} imply, by \cite[Lemma~4.3.10]{Laz04}, that
\begin{equation}
\label{eq:R-vanishing}
R^j\p_*(\omega_Y \otimes \Sym^t\cN_{Y/X})=0
\quad\text{for $j \ge r$ and $t\ge 1$}.
\end{equation}
Then the Leray spectral sequence gives
\begin{equation}
\label{eq:Leray}
H^i(Z,R^j\p_*(\omega_Y\otimes\Sym^t\cN_{Y/X})\otimes\cA^*)
\Rightarrow
H^{n-r-k}(Y,\omega_Y\otimes\Sym^t\cN_{Y/X}\otimes\cL^*)
\end{equation}
for $k=0,1$, where $i+j=n-r-k$. The relative vanishing \eqref{eq:R-vanishing} 
kills the terms on the left-hand side of \eqref{eq:Leray} with $j\ge r$, and those with 
$j < r$ and $i \ge n-2r$ vanish since $\dim Z < n - 2r$. 
Thus, we have $H^{n-r-k}(Y,\omega_Y\otimes\Sym^t\cN_{Y/X}\otimes\cL^*)=0$, 
and \eqref{eq:desired-vanishing} follows.

Consider the following exact sequence
\[
0 \to \omega_{\~X}(tE)\otimes\~\cL^* \to \omega_{\~X}\bigl((t+1)E\bigr)\otimes\~\cL^* \to 
\omega_E(tE)\otimes\~\cL^*|_E \to 0.
\]
By \eqref{eq:desired-vanishing}, we have the bijectivity of
\[
H^{n-k}(\~X,\omega_{\~X}(tE)\otimes\~\cL^*) \to 
H^{n-k}(\~X,\omega_{\~X}\bigl((t+1)E\bigr)\otimes\~\cL^*)
\quad\text{for $t \ge 1$ and $k=0,1$.}
\]
It follows then by \cite[Equation~(5.1)]{Ott12} that
\[
H^{n-j}(\~X,\omega_{\~X}(tE)\otimes\~\cL^*) \cong H^{n-j}(\~X\setminus E,(\omega_{\~X}\otimes\~\cL^*)|_{\~X-E})
\quad\text{for $j=0,1$ and $t \ge 1$}.
\]
Since, by assumption, $\cd(\~X\setminus E)=\cd(X\setminus Y)\leq n-2$, we have
$H^{n-j}(\~X\setminus E,(\omega_{\~X}\otimes\~\cL^*)|_{\~X-E})=0$ for $j=0,1$, hence
$H^{n-j}(\~X,\omega_{\~X}(tE)\otimes\~\cL^*)=0$ for $j=0,1$ and $t \ge 1$.
This implies that the map $H^0(X,\cL) \to H^0(Y,\cL|_Y)$ is bijective.
\end{proof}

\begin{lemma}
\label{prop:unique}
Let $X$ be a projective variety
and $Y \subset X$ an ample subvariety equipped with
a dominant morphism $\p \colon Y \to Z$. 
Assume that $\dim Y > \max\{1,\dim Z\}$.
Then any extension $\~\p \colon X \to Z$ of $\p$, if it exists, is unique. 
\end{lemma}

\begin{proof}
By \cref{l:Stein}, we can take the Stein factorization and reduce to the 
case where $\p$ is an algebraic fiber space. 
Let $\mathcal{A}$ be a very ample line bundle on $Z$ and $\mathcal{L}:=\~\p^*\mathcal{A}$ be its pull-back to $X$.
Since by \cref{l:Pic-iso} the restriction map $\Pic(X)\to \Pic(Y)$ is injective,
the line bundle $\mathcal{L}$ is independent of the choice of extension $\~\p$ of $\p$.
Thanks to the fact that $\~\p$ is an algebraic fiber space, we have the isomorphism $H^0(X,\mathcal{L})\cong H^0(Z,\mathcal{A})$.
Thus, the morphism $\~\p\colon X\to Z$ is determined by the complete linear system $|H^0(X,\mathcal{L})|$ and therefore it is unique.
\end{proof}

\begin{lemma}
\label{l:mapfromA}
Let $A$ be an abelian variety and $\pi\colon A\to W$ an algebraic fiber space 
(i.e., a morphism such that $\O_W \cong \pi_*\O_A$).
Then $W$ is an abelian variety and, up to translation, $\pi$ is a morphism of abelian varieties.
\end{lemma}

\begin{proof}
First, let us show that if $\pi$ is birational then it is an isomorphism.
Suppose $C$ is an irreducible curve contracted by $\pi$.
Let $\mathcal{A}$ be an ample line bundle on $W$ and $\mathcal{L}:=\pi^*\mathcal{A}$ be its pull-back to $A$.
Then $\mathcal{L}|_C$ is trivial.
Let $a$ be an arbitrary closed point of $A$ and 
$t_{-a}\colon A\to A$ be the translation by $-a$.
Identifying $C$ and $a+C$ via the translation map $t^{-a}$,
we have $\mathcal{L}|_{a+C}\cong t^*_{-a}\mathcal{L}|_C$.
Note that $t^*_{-a}\mathcal{L}$ is algebraically equivalent to $\mathcal{L}$,
so $\mathcal{L}|_{a+C}$ is numerically trivial.
Therefore, $a+C$ is contracted by $\pi$ as well.
Since $a$ is arbitrary, this shows that if $\pi$ is birational then
$\pi$ has to be an isomorphism. 

Now assume that $\dim A>\dim W$.
Take the fiber $A_w$ over a general closed point $w\in W$.
Note that $A_w$ is smooth and has trivial normal bundle.
Therefore the tangent bundle of $A_w$ is trivial as well.
It is a fact that a smooth projective variety over an algebraically closed field of characteristic zero with trivial tangent bundle has to be isomorphic to an abelian variety \cite[Page~191]{MS87}.
Therefore, $A_w$ is isomorphic to an abelian variety $B$.
We may identify $B$ with an abelian subvariety of $A$
so that $A_w=a+B$ for some closed point $a\in A$.
We claim that $\pi$ factors through the quotient morphism $A\to A/B$.
To see this, note that the pull-back of any ample line bundle on 
$W$ is numerically trivial on $a'+B$ for any closed point $a'\in A$.
This implies that $a'+B$ is contracted by $\pi$, 
thus $\pi$ factors through $A\to A/B$.
The induced map $A/B\to W$ is birational and therefore is an isomorphism by the above argument.
\end{proof}

\section{Sommese's extendability conjecture}

As in the previous section, we work over an algebraically closed field of characteristic zero.
It was conjectured by Sommese that if $X$ is a smooth projective variety
and $Y \subset X$ is a smooth subvariety of codimension $r$
defined by the vanishing of a regular section of an ample
vector bundle $\cE$ on $X$, 
then any morphism $\p \colon Y \to Z$ with $\dim Y - \dim Z > r$
extends to a morphism $\~\p \colon X \to Z$ (see \cite[Page~71]{Som76}).
Here, we consider the conjecture only assuming that $Y$ is an ample subvariety. Notice that, 
as we discussed before, subschemes defined by regular sections of ample vector bundles are 
ample in the ambient variety.
We restate Sommese's conjecture in the context of ample subvarieties as follows.

\begin{conjecture}
\label{conj:Som}
Let $X$ be a smooth projective variety 
and $Y \subset X$ a smooth ample subvariety of codimension $r$.
Then any morphism $\p \colon Y \to Z$ with $\dim Y - \dim Z > r$
extends uniquely to a morphism $\~\p \colon X \to Z$. 
\end{conjecture}

The case $r=1$ of the conjecture (i.e., when $Y$ is an ample divisor) was proved in \cite{Som76}.
The condition that $\dim Y - \dim Z > r$ is sharp.
When $r=1$, this is discussed in \cite[Section~3]{BI09}, and 
the construction given there is generalized to arbitrary codimension $r$, see \cite{dFL}.

\cref{th:X-abelian} shows that \cref{conj:Som} holds when $X$ is an abelian variety. 
By a similar argument, which in fact extends the proof of \cite[Proposition~III]{Som76} (see in particular
the proof given in \cite[Theorem~3.1]{BI09}), we obtain the following general
condition for a morphism to extend.

\begin{proposition}
\label{th:Kodaira-type}
Let $X$ be a smooth projective variety 
and $Y \subset X$ a smooth ample subvariety of codimension $r$.
Assume that for any globally generated line bundle $\cM$ on $Y$ we have
\begin{equation}
\label{eq:Kodaira-type-2}
H^i(Y,\omega_Y \otimes \Sym^t \cN_{Y/X} \otimes \cM)=0
\quad\text{for $i \ge r$, $t \ge 1$.}
\end{equation}
Then any morphism $\p \colon Y \to Z$ with $\dim Y - \dim Z > r$
extends uniquely to a morphism $\~\p \colon X \to Z$. 
\end{proposition}

\begin{proof}
As the statement is trivial if $Z$ is a point, 
we may assume that $\dim Z \ge 1$ and hence that $\dim Y \ge 3$.
After taking Stein factorization, we may also assume that the natural map $\O_Z\to\pi_*\O_Y$
is an isomorphism.
Let $\cA$ be a very ample line bundle on $Z$.
By projection formula, we have $H^{0}(Y,\pi^{*}\cA)\cong H^{0}(Z,\cA)$.
The composition $Y\to Z\to\mathbb{P}(H^{0}(Z,\cA))$ is just given by the 
isomorphism $H^{0}(Z,\cA)\cong H^{0}(Y,\pi^{*}\cA)$.

By \cref{l:Pic-iso}, after possibly replacing $\cA$ by a positive multiple, 
there exists a line bundle $\cL$ on $X$ such that $\cL|_Y \cong \p^*\cA$. 
By \cref{l:H^0(L)=H^0(L|Y)}, the restriction map 
$H^0(X,\cL) \to H^0(Y,\cL|_Y)$ is bijective. 

We shall now show that $\cL$ is generated by global sections.
We see by \eqref{eq:bij} that 
the base locus $B$ of the linear system $|H^0(X,\cL)|$ is disjoint from $Y$.
Suppose that $B\neq \emptyset$.
Note that $(\dim Z+1)$-many general sections in $H^0(Y,\cL|_Y)$ have no common zeroes.
Lifting these sections to $H^0(X,\cL)$, 
they cut out a closed subscheme $B'$ of $X$ that contains $B$, is disjoint from $Y$, and 
has dimension $\ge \dim X-\dim Z -1$.
Since, by \cite[Theorem~5.4]{Ott12}, $X \setminus Y$  
has cohomological dimension $r -1$,
it cannot contain a projective subscheme of dimension greater than $r-1$.
Since by our hypothesis we have
$\dim B' > r-1$, this gives the desired contradiction.

The last part of the proof of \cref{th:X-abelian} can be repeated here verbatim
to conclude that the complete linear system $|X^0(X,\cL)|$ defines the desired extension $\~\p \colon X \to Z$.
The unicity of $\~\p$ follows from \cref{prop:unique}. 
\end{proof}

\begin{remark}
\label{r:R-supp-dim}
The vanishing condition in \eqref{eq:Kodaira-type-2} is only needed to ensure
the vanishing of higher direct images in \eqref{eq:R-vanishing}, 
so the hypothesis of the \lcnamecref{th:Kodaira-type} can be 
relaxed by only assuming the latter. 
In fact, all we need is the vanishing of 
the left-hand side of \eqref{eq:Leray}. 
Therefore, the vanishing hypothesis in \cref{th:Kodaira-type}
can be replaced with the weaker condition that
\begin{equation}
\label{eq:R-supp-dim}
\dim \Supp\big(R^j\p_*(\omega_Y \otimes \Sym^t\cN_{Y/X})\big) < \dim X-r-j-1
\quad\text{for $j \ge r$, $t\ge 1$.}
\end{equation}
\end{remark}

Using the above properties, we can verify \cref{conj:Som} when $\p$ is special, 
in the following sense. 
We fix an algebraically closed field $k$ of characteristic zero, 
and define two classes $\cU$ and $\cV$ of varieties over $k$ whose union contains
abelian varieties and toric varieties. 
The first class, $\cU$, was introduced by Manivel in \cite{Man96}. 
The key notion is that of \emph{uniformly nef vector bundle}, 
for which we refer to \cite[Section~2.2]{Man96}; roughly speaking, the category 
$\cC$ of uniformly nef vector bundles is the smallest subcategory of the category 
of vector bundles on $k$-varieties which contains all bundles of the form $\cE \otimes \cL$ where
$\cE$ is a Hermitian flat vector bundle and $\cL$ is a nef line bundle, that 
is closed under quotient, extension, and direct sum decomposition, 
and that satisfies the condition that for a finite morphism $f \colon Y \to X$
and a vector bundle $\cE$ on $X$, one has $\cE \in \cC$ if and only if $f^*\cE \in \cC$. 
One then defines $\cU$ to be the class of smooth projective varieties over $k$ 
with uniformly nef tangent bundle. 

\begin{example}
\label{eg:X-with-unif-nef-TX}
Prototypes of smooth projective varieties with uniformly nef tangent bundle
are projective spaces and abelian varieties.
More examples can be constructed starting from these using the fact that 
the class $\cU$ is closed under products, finite \'etale covers, and the following construction:
given $X \in \cU$ and $\cE_1,\dots,\cE_m$ are numerically flat vector bundles on $X$, 
we have $\P(\cE_1) \times_X \dots \times_X \P(\cE_m) \in \cU$; see \cite[Section~2.3]{Man96}.
\end{example}

The second class of varieties, which we denote by $\cV$, 
consists of those smooth projective varieties $X$ over $k$
which admit an arithmetic thickening $X_A \to \Spec A$ 
with a dense set of fibers with the \emph{$F$-liftability property}, by which we mean that 
for a dense set of closed points $p \in \Spec A$
there exists a lift of the absolute Frobenius of $(X_A)_p$ to the second Witt vectors
$W_2(k(p))$.
Here, by \emph{arithmetic thickening} of a scheme $X$ over a field $k$ of
characteristic zero, we mean a choice of a finitely generated $\Z$-subalgebra
$A \subset k$ and a flat scheme $X_A$ over $\Spec A$ such that $X \cong X_A \times_{\Spec A} \Spec k$.

\begin{example}
\label{eg:sX-with-F-lift-property}
Examples of varieties in the class $\cV$ are toric varieties, abelian varieties
whose arithmetic thickenings contain 
a dense set of ordinary abelian varieties as fibers,
\'etale quotients of abelian varieties with the above property, 
and toric fibratons over such abelian varieties
\cite[Examples~3.1.2--3.1.5]{AWZ17}.
It is expected that all abelian varieties satisfy the above property, 
see \cite[Conjecture~1.1 and Example~5.4]{MS11}.
\end{example}

\begin{proposition}
\label{th:special-Y}
Let $X$ be a smooth projective variety and
$Y \subset X$ a smooth ample subvariety of codimension $r$. Suppose that
$Y$ belongs to the union $\cU \cup \cV$ of the two classes defined above. 
Then any morphism $\p \colon Y \to Z$ with $\dim Y - \dim Z > r$
extends uniquely to a morphism $\~\p \colon X \to Z$. 
\end{proposition}

\begin{proof}
The result follows from \cref{th:Kodaira-type}.
If $Y$ is toric or is a variety with uniformly nef tanget bundle, then 
the necessary vanishing follows from \cite[Th\'eor\`eme~2.5 and Th\'eor\`eme~5.3]{Man96};
the statement of \cite[Th\'eor\`eme~5.3]{Man96} does not include
the twist by a nef line bundle, but the proof extends to cover that case.
For the remaining cases, we apply \cite[Theorem~2.2.1]{Lit17b} (with $j=1$)
along a dense set of fibers over an arithmetic thickening and
use upper semicontinuity of cohomology, 
relying on \cite[Theorem~3]{Ara04}
to allow for the twist of a globally generated line bundle.
\end{proof}

\begin{definition}
We say that a
surjective morphism of varieties $f \colon V \to W$
\emph{does not contract divisors} if there are no prime divisors
$D$ in $V$ such that $f(D)$ has codimension $\ge 2$ in $W$.
\end{definition}

\begin{proposition}
\label{th:fiber-pi-special}
Let $X$ be a smooth projective variety,
$Y \subset X$ a smooth ample subvariety of $X$ of codimension $r$, 
and $\p \colon Y \to Z$ a surjective morphism with $\dim Y - \dim Z > r$.
Suppose that $\pi$ does not contract divisors
and that there exists an open set $Z^* \subset Z$ with complement of
codimension $\ge 2$ such that $\pi$ restricts
to a smooth family $Y^* \to Z^*$ of varieties belonging to $\cU \cup \cV$.
Then $\p$ extends uniquely to a morphism $\~\p \colon X \to Z$. 
\end{proposition}

\begin{proof}
Here we apply \cite[Th\'eor\`eme~2.5 and Th\'eor\`eme~5.3]{Man96} and
\cite[Theorem~2.2.1]{Lit17b} along the fibers of $\p$ over $Z^*$ 
as in the proof of \cref{th:special-Y}, 
and use the hypothesis that $\p$ does not contract divisors
to check the condition in \eqref{eq:R-supp-dim}. 
\end{proof}

\begin{corollary}
\label{c:Y=AxB}
Let $Y = A \times B$ where $A$ and $B$ are two varieties in $\cU \cup \cV$. 
If $Y$ can be embedded as an ample subvariety 
of codimension $r$ of a smooth projective variety $X$, 
then $\min\{\dim A, \dim B\} \le r$.
\end{corollary}

\begin{proof}
If $\min\{\dim A, \dim B\} > r$, then both projections of $A \times B$
extend to $X$ by \cref{th:fiber-pi-special} and we get
a retraction $\r \colon X \to Y$. Starting from a
sufficiently divisible line bundle $\cL$ on $X$, we have
$\r^*(\cL|_Y) \cong \cL$ by \cref{l:Pic-iso}, 
which is impossible if $\cL$ is ample. 
\end{proof}

\begin{corollary}
\label{c:abelian-fibrations}
Let $Y$ be a smooth projective variety.
Suppose that $\pi \colon Y \to Z$ is a 
smooth family of abelian varieties of dimension $d \ge 2$.
Then $Y$ cannot be realized as an ample subvariety of codimension 
$r < d$ in any smooth projective variety. 
\end{corollary}

\begin{proof}
Suppose by contradiction that $Y$ is an ample subvariety of codimension 
$r < d$ of a smooth projective variety $X$. 
By \cref{th:fiber-pi-special}, $\p$ extends to 
a morphism $\~\p \colon X \to Z$. Let $z \in Z$ be a general closed point
and $Y_z$ and $X_z$ the respective fibers over $z$. 
By \cite[Proposition~4.8]{Lau16}, $Y_z$ is an ample subvariety of $X_z$. 
Since $Y_z$ is an abelian variety of dimension $d \ge 2$, 
this contradicts \cref{th:Y=A-cannot-be-ample}.
\end{proof}


\end{document}